\documentclass{article}

\usepackage{amssymb, latexsym,pdfsync,amsmath,amsthm, ulem}
\newtheorem{theorem}{Theorem}

\newtheorem{lemma}{Lemma}

\newtheorem{definition}{Definition}
\newtheorem{remark}{Remark}
\newtheorem{example}{Example}

\newcommand{\SSS}{\mathbb{S}}
\newcommand{\NN}{\mathbb{N}}

\newcommand{\FF}{\mathbb{F}}

\newcommand{\Fq}{\mathbb{F}_q}
\newcommand{\Fqn}{\mathbb{F}_{q^n}}
\newcommand{\Fqm}{\mathbb{F}_{q^m}}

\newcommand{\cK}{\mathcal K}
\newcommand{\cB}{\mathcal B}

\newcommand{\D}{\mathcal D}
\newcommand{\cS}{\mathcal S}

\def\S{\mathbb{S}}

\def\F{\mathbb{F}}
\def\Fq{{\mathbb{F}}_q}

\def\End{\mathrm{End}}
\def\Aut{\mathrm{Aut}}

\def\PG{\mathrm{PG}}

\def\GammaL{\mathrm{\Gamma L}}
\def\dim{\mathrm{dim}}
\def\im{\mathrm{im}}

\def\id{\mathrm{id}}

\def\tr{\mathrm{tr}}
\def\GTF{\mathrm{GTF}}

\newcommand{\npmatrix}[1]{\left[ \begin{matrix} #1 \end{matrix} \right]}

\begin{document}
\title{On BEL-configurations and finite semifields}
\author{Michel Lavrauw and John Sheekey}
\date{\today}
\maketitle

\begin{abstract}
The BEL-construction for finite semifields was introduced in \cite{BEL2007}; a geometric method for constructing semifield spreads, using so-called BEL-configurations in $V(rn,q)$. In this paper we investigate this construction in greater detail, and determine an explicit multiplication for the semifield associated with a BEL-configuration in $V(rn,q)$, extending the results from \cite{BEL2007}, where this was obtained only for $r=n$.
Given a BEL-configuration with associated semifields spread $\cS$, we also show how to find a BEL-configuration corresponding to the dual spread $\cS^d$. 
Furthermore, we study the effect of polarities in $V(rn,q)$ on BEL-configurations, leading to a characterisation of BEL-configurations associated to symplectic semifields.

We give precise conditions for when two BEL-configurations in $V(n^2,q)$ define isotopic semifields. We define operations which preserve the BEL property, and show how non-isotopic semifields can be equivalent under this operation. We also define an extension of the ```switching'' operation on BEL-configurations in $V(2n,q)$ introduced in \cite{BEL2007}, which, together with the transpose operation, leads to a group of order $8$ acting on BEL-configurations.
\end{abstract}


\section{Introduction}

A \emph{finite semifield} $(\SSS ,+,\circ)$ is a finite division algebra except that associativity of multiplication 
is not assumed. Precisely: $(\SSS,+)$ is an abelian group; both distributive laws hold; $(\SSS,\circ)$ has no
zero divisors and has an identity. If we do not assume a multiplicative identity element, the structure is known as a \emph{presemifield}.  When there is no confusion possible, we will just write $\SSS$ instead of $(\SSS,+,\circ)$, and
we will conveniently choose between the notations $\SSS(x,y)$ and $x \circ y$ for the multiplication of two elements $x,y\in \SSS$.

The first proper finite semifields were constructed by Dickson in 1906. Albert defined an important class of semifields known as \emph{generalized twisted fields}. Semifields play a key role in the study of projective planes, as they correspond to translation planes which are also dual translation planes, and they are also related to
various other structures from finite geometry and the theory of finite fields.
For more background, history, and known classifications, see for example \cite{Knuth1965}, \cite{Kantor2006}, \cite{LaPo2011} and \cite{Lavrauw2013}. Most of the remainder of this section is standard knowledge in the subject, but we reproduce it here for clarity of exposition and to establish notation and conventions.

To each semifield there are several important substructures, all of which are isomorphic to finite fields. The {\it
left nucleus}\index{left nucleus} ${\mathbb{N}}_l({\mathbb{S}})$, 
{\it the middle nucleus}\index{middle nucleus}
${\mathbb{N}}_m({\mathbb{S}})$, and the {\it right nucleus}\index{right nucleus}
${\mathbb{N}}_r({\mathbb{S}})$ are defined as follows: 
\begin{align*}
{\mathbb{N}}_l({\mathbb{S}})&:=\{x~:~ x \in {\mathbb{S}} ~|~ x \circ (y\circ
z)=(x\circ y)\circ z, ~\forall y,z \in {\mathbb{S}}\}, \\
{\mathbb{N}}_m({\mathbb{S}})&:=\{y~:~ y \in {\mathbb{S}} ~|~ x \circ (y\circ
z)=(x\circ y)\circ z, ~\forall x,z \in {\mathbb{S}}\}, \\
{\mathbb{N}}_r({\mathbb{S}})&:=\{z~:~ z \in {\mathbb{S}} ~|~ x \circ
(y\circ z)=(x\circ y)\circ z, ~\forall x,y \in {\mathbb{S}}\}.
\end{align*}
The intersection $\NN(\SSS)$ of the nuclei is called the \emph{associative centre}\index{associative center}, and the elements of $\NN(\SSS)$ which commute with all other elements of $\SSS$ form the \emph{centre}\index{center} $Z(\SSS)$. Then $\SSS$ has the structure of a left vector space over $\NN_l(\SSS)$, which we may denote by $V_l(\SSS)$. Similarly $\SSS$ has the structure of a left and right vector space over $\NN_m(\SSS)$, and a right vector space over $\NN_r(\SSS)$.

Let $\SSS$ be an $n$-dimensional semifield over $\Fq$, i.e. a semifield of order $q^n$ with centre containing $\Fq$. We identify the elements of $\SSS$ with the elements of $\Fqn$. It follows from the definition that there exist unique elements $c_{ij} \in \Fqn$ such that
\[
\SSS(x,y) = \sum_{i,j=0}^{n-1} c_{ij} x^{q^i}y^{q^j}.
\]
%For example, Albert's generalized twisted field are defined as follows. Take the elements of $\SSS$ to be the elements of $\Fqn$. Let $i,j \in \{0,\ldots,n-1\}$, and $c \in \Fqn^{\times}$ such that $c \notin \{x^{q^i-1}y^{q^j-1}:x,y \in \Fqn\}$. Then define
%\[
%\mathrm{GTF}_{c,i,j}(x,y) = xy - cx^{q^i}y^{q^j}.
%\]

For example, Albert's generalized twisted fields are defined as follows. Let $\alpha, \beta \in \Aut(\Fqn:\Fq)$, and $c \in \Fqn$ such that $c \notin \{x^{\alpha-1}y^{\beta-1}:x,y \in \Fqn\}$. Then the multiplication on
$\Fqn$
\[
x\circ y = xy - cx^{\alpha}y^{\beta}
\]
defines a presemifield of order $q^n$. We will denote this presemifield by $\mathrm{GTF}_{c,\alpha,\beta}$.
If $x^{\alpha} = x^{q^k}$ and $y^{\beta} = y^{q^m}$, we see that
\[
c_{ij} = \left\{\begin{array}{ll}
1& \textrm{if $i=j=0$},\\
-c & \textrm{if $i=k,j=m$},\\
0 & \textrm{otherwise}.
\end{array}\right.
\]

Continuing with the semifield $\SSS$ as before we have that each $y\in \SSS$ defines an $\Fq$-endomorphism of $\SSS$, denoted by $R_y$, defined by
\begin{eqnarray}
\label{eqn:rightmult1}
R_y(x) := \SSS(x,y).
\end{eqnarray}

We call this the endomorphism of right multiplication by $y$. Since $\SSS$ has no zero divisors, $R_y$ is nonsingular for each $y \ne 0$, and the set $R(\SSS) := \{R_y :y \in \SSS\}$ is an $\Fq$-subspace of $\Fq$-endomorphisms of $\SSS$, where each nonzero element is nonsingular. We call $R(\SSS)$ the \emph{spread set} of $\SSS$. %Note that each $R_y$ can also be viewed as an $\NN_l(\SSS)$-endomorphism of $V_l(\SSS)$.

Similarly we can define the endomorphisms of left multiplication
\[
L_x(y) := \SSS(x,y).
\]
If we define $r_i(y) = \sum_j c_{ij} y^{q^j}$, and $l_j(x) = \sum_i c_{ij}x^{q^i}$, then we have that
\begin{align}
\label{eqn:rightmult2}
R_y(x) &= \sum_i \left(\sum_j c_{ij} y^{q^j}\right) x^{q^i}= \sum_i r_i(y)x^{q^i};\\
L_x(y) &= \sum_j \left(\sum_i c_{ij} x^{q^i}\right) x^{q^j}= \sum_j l_j(x)y^{q^j}.
\end{align}

An \emph{$n$-spread} $\cS$ in $V(rn,q)$ is a set of subspaces of dimension $n$ which pairwise intersect trivially, and which partition the nonzero elements of $V(rn,q)$. It is well known that every (pre)semifield defines a spread in the following way. We represent the elements of $V(2n,q)$ by elements of $\SSS^2$. Then for each $y \in \SSS$, we define
\[
A_y := \{(x,\SSS(x,y)):x \in \S\} = \{(x,R_y(x)):x \in \S\}.
\]
We also define
\[
A_{\infty} = \{(0,x):x \in \SSS\}.
\]
Then the set of subspaces of dimension $n$
\begin{eqnarray}
\label{def:semspread}
\cS(\SSS) := \{A_y:y \in \SSS\} \cup \{A_{\infty}\} 
\end{eqnarray}
defines a spread.

Conversely, every semifield spread defines a (pre)semifield. A spread $\cS$ is a \emph{semifield spread} if there exists some element $T \in \cS$, and a group $G \leq \GammaL(2n,q)$ fixing $T$ pointwise and acting transitively on the other elements of $\cS$. Note that in the spread $\cS(\SSS)$ the special element is $S_\infty$.

Two semifields $\SSS$ and $\SSS'$ are \emph{isotopic} if there exist nonsingular linear maps $A,B,C:\SSS \rightarrow \SSS'$ such that
\[
\SSS'(A(x),B(y)) = C(\SSS(x,y))
\]
for all $x,y \in \SSS$. We denote the \emph{isotopy class} of $\SSS$ by $[\SSS]$.

A well known result of Albert \cite{Albert1960} says that two (pre)semifields are isotopic if and only if the spreads they define are equivalent (under the action of $\GammaL(2n,q)$). For equivalent spreads $\cS_1$ and $\cS_2$ we write $\cS_1\simeq \cS_2$.

It was shown in \cite{Knuth1965} that from each (pre)semifield, we can obtain a chain of (up to) six isotopy classes, which we call the \emph{Knuth orbit} and denote by $\cK(\SSS)$. These are obtained via an action of the symmetric group $S_3$ on the set of isotopism classes of semifields, generated by the two operations \emph{dual}, denoted by $[\SSS]^d$, and \emph{transpose}, denoted by $[\SSS]^t$. 

The action of $S_3$ in \cite{Knuth1965} was defined using the cubical array of the structure constants of the semifield, but here we give the alternative more geometric definition of the Knuth orbit.
The dual of $[\SSS]$ corresponds to the dual projective plane and the opposite algebra $\SSS^d$, that is
\[
\SSS^d(x,y) = \SSS(y,x)
\]
and $[\SSS]^d:=[\SSS^d]$. 
The transpose $[\SSS]^t$ of $[\SSS]$ can be defined using the spread $\cS(\SSS)$, and corresponds to the dual spread with respect to some nondegenerate symmetric bilinear form on $V(2n,q)$. Note that different choices of forms give equivalent spreads. The spread set of $[\SSS]^t$ then consists of the adjoint linear transformations with respect to some nondegenerate symmetric bilinear form on $V(n,q)$. For a particular choice of form and basis, this adjoint becomes the transpose operation on matrices, which explains the notation and the name. See for example \cite{Knuth1965} for this. Here we will choose a different form, which better suits our needs for this paper. 

We define the form $b_{\epsilon}$ on $V(2n,q)$ by
\begin{eqnarray}
b_{\epsilon}((a,b),(c,d)) = \tr(ad-bc),
\end{eqnarray}
where $\tr$ denotes the trace function from $\Fqn$ to $\Fq$, that is 
\[
\tr(x) = x+x^q+\ldots+x^{q^{n-1}}.
\] 
Given a subspace $M$, define the dual of $M$, denoted by $M^{\epsilon}$, by
\begin{eqnarray}
\label{def:epsilon}
M^{\epsilon} = \{v:v \in V(2n,q)\mid b_{\epsilon}(u,v)=0 ~\forall u \in M \}.
\end{eqnarray}

Let $f$ be an $\Fq$-endomorphism of $\Fqn$. It is well known that 
\[
f(x) = \sum_{i=0}^{n-1} f_i x^{q^i}
\]
for some unique $f_i \in \Fqn$. We denote by $\hat{f}$ the adjoint of $f$ with respect to the nondegenerate symmetric bilinear form on $V(n,q)$ defined by
\[
(x,y) \mapsto \tr(xy).
\]
That is, $\hat{f}$ is the unique endomorphism such that
\[
\tr(f(x)y) = \tr(x\hat{f}(y))
\]
for all $x,y \in \Fqn$. Then we have that
\[
\hat{f}(x) = \sum_{i=0}^{n-1} f_{n-i}^{q^i} x^{q^i}.
\]
It is not difficult to see that $\hat{\hat{f}}=f$, and $\widehat{fg}= \hat{g}\hat{f}$ for all endomorphisms $f,g$.

The isotopism class $[\SSS]^t$ corresponds to the dual spread of $\cS(\SSS)$, that is $\cS(\SSS)^{\epsilon} = \{T^{\epsilon}:T \in \cS(\SSS)\}$. Note that $A_{\infty}^{\epsilon} = A_{\infty}$.

Now we can see that
\[
A_y^{\epsilon} = \{(x,\hat{R}_y(x)):x \in \Fqn\},
\]
since
\begin{align*}
b_{\epsilon}((x,R_y(x)),(z,\hat{R}_y(z))) &= \tr(x\hat{R}_y(z) - R_y(x)z)\\
		&= \tr(z(R_y(x) - R_y(x))\\
		&= 0
\end{align*}
for all $x,z \in \Fqn$.

This allows us to define the transpose of $[\SSS]$ directly from the multiplication $\SSS(x,y)$.
We have shown the following.
\begin{lemma}
Let $\SSS$ be a semifield. Then $[\SSS]^t=[\SSS^t]$ where
\begin{equation}
\label{eqn:tmult}
\SSS^t(x,y) := \hat{R}_y(x) = \sum_i (r_{n-i}(y))^{q^i}x^{q^i},
\end{equation}
where $R_y$ and $r_i$ are as defined in (\ref{eqn:rightmult1}) and (\ref{eqn:rightmult2}).
\end{lemma}

In \cite{Knuth1965} Knuth also showed that these operations are well defined up to isotopism. This fact is
easily verified using the geometric description of the transpose and dual operations given above, and goes
as follows.

Suppose $\SSS$ is isotopic to $\SSS'$, i.e. there exist a triple of invertible linear transformations $(A,B,C)$ such that $\SSS'(A(x),B(y)) = C(\SSS(x,y))$ for all $x,y \in \SSS$. Then 
\[
C(\SSS^d(x,y)) = C(\SSS(y,x)) = \SSS'(A(y),B(x)) = \SSS'^d(B(x),A(y)),
\]
and hence $\SSS^d$ is isotopic to $\SSS'^d$, with corresponding isotopism $(B,A,C)$.

Now $CR_y(x) = C(\SSS(x,y)) = \SSS'(A(x),B(y)) = R'_{B(y)}A(x)$ for all $x,y\in \SSS$, and hence $CR_y = R'_{B(y)}A$ for all $y\in \SSS$. Taking the adjoint of both sides, we get $\hat{R}_y \hat{C} = \hat{A} \hat{R}'_{B(y)}$. Hence
\[
\hat{A}^{-1}(\SSS^t(x,y)) = \hat{A}^{-1}\hat{R}_y(x) = \hat{R}'_{B(y)}\hat{C}^{-1}(x) = \SSS'^t(\hat{C}^{-1}(x),B(y)),
\]
implying that $\SSS^t$ is isotopic to $\SSS'^t$, as claimed. The corresponding isotopism is 
$(\hat{C}^{-1},B,\hat{A}^{-1})$.

It is clear from the above that these two operations satisfy $t^2=d^2=\id$, $tdt=dtd$, and so form a group isomorphic to the symmetric group $S_3$. Hence we get a chain of (up to) $6$ isotopism classes. The set of these isotopism classes is called the Knuth orbit of a semifield $\SSS$:
\begin{eqnarray}
\cK(\SSS)=\{[\SSS],[\SSS^t],[\SSS^d],[\SSS^{td}],[\SSS^{dt}],[\SSS^{tdt}]=[\SSS^{dtd}]\}.
\end{eqnarray}
We summarise a multiplication for each isotopism class in the Knuth orbit of $\SSS$ in the below table. We give two equivalent expressions for convenience. Here indices such as $-i$ are understood to be modulo $n$, and all sums are understood to be over all $i,j \in \{0,\ldots ,n-1\}$.

\begin{center}
\begin{tabular}{|c|c|c|}
\hline
$\SSS$ & $\sum c_{ij} x^{q^i}y^{q^j}$& $\sum c_{ij} x^{q^i}y^{q^j}$\\
\hline
$\SSS^t$ & $\sum c_{-i,i-j}^{q^i} x^{q^i} y^{q^j}$& $\sum c_{ij}^{q^{-i}} x^{q^{-i}}y^{q^{j-i}}$\\
\hline
$\SSS^d$ & $\sum c_{ji} x^{q^i}y^{q^j}$& $\sum c_{ij} x^{q^j}y^{q^i}$\\
\hline
$\SSS^{td}$ & $\sum c_{-j,j-i}^{q^j} x^{q^i} y^{q^j}$& $\sum c_{ij}^{q^{-i}} x^{q^{j-i}}y^{q^{-i}}$\\
\hline
$\SSS^{dt}$ & $\sum c_{j-i,-i}^{q^i} x^{q^i} y^{q^j}$& $\sum c_{ij}^{q^{-j}} x^{q^{-j}}y^{q^{i-j}}$\\
\hline
$\SSS^{dtd}$ & $\sum c_{i-j,-j}^{q^j} x^{q^i} y^{q^j}$& $\sum c_{ij}^{q^{-j}} x^{q^{i-j}}y^{q^{-j}}$\\
\hline
\end{tabular}
\end{center}

\begin{example}
If $\SSS = \mathrm{GTF}_{c,\alpha,\beta}$, then the Knuth orbit is represented in the following table.

\begin{center}
\begin{tabular}{|c|c|c|}
\hline
& Multiplication & GTF\\
\hline
$\SSS$ & $xy - c x^{\alpha}y^{\beta}$&$(c,\alpha,\beta)$\\
\hline
$\SSS^t$ & $xy - c^{1/\alpha} x^{1/\alpha}y^{\beta/\alpha}$&$(c^{1/\alpha},1/\alpha,\beta/\alpha)$\\
\hline
$\SSS^d$ & $xy -  c x^{\beta}y^{\alpha}$&$(c,\beta,\alpha)$\\
\hline
$\SSS^{td}$ & $xy - c^{1/\alpha} x^{\beta/\alpha}y^{1/\alpha}$&$(c^{1/\alpha},\beta/\alpha,1/\alpha)$\\
\hline
$\SSS^{dt}$ & $xy - c^{1/\beta} x^{1/\beta}y^{\alpha/\beta}$&$(c^{1/\beta},1/\beta,\alpha/\beta)$\\
\hline
$\SSS^{dtd}$ & $xy - c^{1/\beta} x^{\alpha/\beta}y^{1/\beta}$&$(c^{1/\beta},\alpha/\beta,1/\beta)$\\
\hline
\end{tabular}
\end{center}

%
%\begin{center}
%\begin{tabular}{|c|c|c|}
%\hline
%& Multiplication & GTF\\
%\hline
%$\SSS$ & $xy - c x^{\alpha}y^{\beta}$&$(c,\alpha,\beta)$\\
%\hline
%$\SSS^t$ & $xy - c^{\alpha^{-1}} x^{\alpha^{-1}}y^{\alpha^{-1}\beta}$&$(c^{\alpha^{-1}},\alpha^{-1},\alpha^{-1}\beta)$\\
%\hline
%$\SSS^d$ & $xy -  c x^{\beta}y^{\alpha}$&$(c,\beta,\alpha)$\\
%\hline
%$\SSS^{td}$ & $xy - c^{\alpha^{-1}} x^{\alpha^{-1}\beta}y^{\alpha^{-1}}$&$(c^{\alpha^{-1}},\alpha^{-1}\beta,\alpha^{-1})$\\
%\hline
%$\SSS^{dt}$ & $xy - c^{\beta^{-1}} x^{\beta^{-1}}y^{\beta^{-1}\alpha}$&$(c^{\beta^{-1}},\beta^{-1},\beta^{-1}\alpha)$\\
%\hline
%$\SSS^{dtd}$ & $xy - c^{\beta^{-1}} x^{\beta^{-1}\alpha}y^{\beta^{-1}}$&$(c^{\beta^{-1}},\beta^{-1}\alpha,\beta^{-1})$\\
%\hline
%\end{tabular}
%\end{center}
\end{example}
In \cite{BiJhJo1999}, the question of when two generalized twisted fields are isotopic was answered. We state this result here, as we will need it in later sections.
\begin{theorem}[\cite{BiJhJo1999}]\label{thm:isotopism_GTF}
The semifields $\GTF_{c,\alpha,\beta}$ and $\GTF_{c',\alpha',\beta'}$ are isotopic if and only if either
\begin{itemize}
\item[(i)] $\alpha=\alpha'$, $\beta=\beta'$ and $c^\rho=c'a^{1-\alpha}b^{1-\beta}$ for some $\rho \in \Aut(\F_{q^n})$, $a,b \in \F_{q^n}^*$, or
\item[(ii)] $\alpha'=1/\alpha$, $\beta'=1/\beta$ and $c'^{\rho}=c^{-1}a^{1-\alpha} b^{1-\beta}$, for some $\rho \in \Aut(\F_{q^n})$, $a,b \in \F_{q^n}^*$.
\end{itemize}

%\begin{itemize}
%\item[(i)] $(c',\alpha',\beta') = (c^{\rho} a^{-\alpha^{-1}\rho} b^{-\beta^{-1}\rho},\alpha,\beta)$, for some $\rho \in \Aut(\F_{q^n})$, $a,b \in \F_{q^n}^*$, or
%\item[(ii)] 
%$(c',\alpha',\beta') = (c^{-\rho}a^{\rho-\alpha\rho} b^{\rho-\beta\rho},\alpha^{-1},\beta^{-1})$, for some $\rho \in \Aut(\F_{q^n})$, $a,b \in \F_{q^n}^*$.
%\end{itemize}

\end{theorem}

\section{BEL-construction}

%Let $U$ be an $(n-1)$-space in $\PG(rn-1,q)$. Then there exist $F=\Fq$-linear maps on $L=\Fqn$, say $f_1,\ldots,f_r$, such that
%\[
%U = \{F(f_1(x),\ldots,f_r(x)):x \in L^*\}.
%\] 
%Note that we must have $\cap_i \ker(f_i) = \{0\}$, for otherwise $U$ would have dimension less than $n-1$. We will write $U=U_f$, where $f = (f_1,\ldots,f_r)$.
%
%Let $W$ be an $(nr-n-1)$-space in $\PG(rn-1,q)$. Then there exist $F$-linear maps on $L$, say $g_1,\ldots,g_r$, such that
%\[
%W = \{F(x_1,\ldots,x_r):x_i \in L|\sum_i g_i(x_i) = 0\}.
%\] 
%We will write $W=W_g$, where $g = (g_1,\ldots,g_r)$.
%
%\subsection{Bel construction}

%A \emph{BEL-configuration} is a triple $(\D,U,W)$, where $\D$ is a Desarguesian spread of $V(rn,q)$, and $U$ and $W$ are subspaces of $V(rn,q)$ satisfying the properties outlined below in Definition \ref{def:BEL}.

The concept of a \emph{BEL-configuration} was introduced in \cite{BEL2007}, and further developed in \cite{Lavrauw2008} and \cite{Lavrauw2011}. In this section we will recall the definition, and how a BEL-configuration can be used to construct of a semifield spread. We restrict ourselves to BEL-configurations where both subspaces $U$ and $W$ are subspaces as in the original paper \cite{BEL2007}, and will not consider the generalization  \cite{Lavrauw2011} to linear sets. We obtain an explicit formula for the pre-semifield multiplication obtained from a BEL-configuration in $V(rn,q)$.

For a spread $\D$ and a subset $A$ of a vector space $V$, we define the sets
\begin{align*}
\cB_{\D}(A) &= \{S \in \D \mid S \cap A \ne \{0\}\};\\
\tilde{\cB}_{\D}(A) &=  \bigcup_{S \in \cB(A)} \{x:x \in S^{\times}\}.
\end{align*}
That is, $\cB_{\D}(A)$ is the set of elements of $\D$ intersecting $A$ nontrivially, and $\tilde{\cB}_{\D}(A)$ is the set of nonzero vectors contained in those elements. We will omit the subscript $\D$ whenever there is no ambiguity. If $0 \ne v \in V$, we denote by $\cB_{\D}(v)$ the unique element of $\D$ containing $v$. Then by definition, $\D = \{\cB_{\D}(v):v \in V^{\times}\}$.
 
A \emph{Desarguesian spread} is a spread obtained by field reduction, see e.g. \cite{LaVdV20??}.

\begin{definition}
\label{def:BEL}
A \emph{BEL-configuration} in $V(rn,q)$ is a triple $(\D,U,W)$ such that:
\begin{itemize}
\item
$\D$ is a Desarguesian $n$-spread;
\item
$U$ is a subspace of dimension $n$;
\item
$W$ is a subspace of dimension $rn-n$;
\item
$\cB(U) \cap \cB(W) = \emptyset$.
\end{itemize}
\end{definition}

We now choose a specific Desarguesian spread $\D_{r,n,q}$, which we will consider fixed for the remainder of this paper. 
We represent elements of $V(rn,q)$ by elements of $(\Fqn)^r$. 
 
For a vector $v = (v_1,\ldots,v_r)\ne 0$, $v_i \in \Fqn$, we define the vector subspace $\cB(v)$ of dimension $n$ in $V(rn,q)$ as follows:
\[
\cB(v) := \{(\alpha v_1,\alpha v_2,\ldots,\alpha v_r):\alpha \in \Fqn\} \leq V(rn,q).
\]
We then take the set of all such subspaces $\D_{r,n,q} = \{\cB(v):v \in V^{\times}\}$, which forms a Desarguesian spread. We will omit the subscripts when there is no ambiguity.
Obtaining a spread in this way is an example of the technique known as \emph{field reduction}. See \cite{LaVdV20??} for a detailed exposition of this.

The following lemma is trivial, but we write it out explicitly for future reference.
\begin{lemma}
\label{lem:belprop}
The following are equivalent:
\begin{enumerate}
\item
$(\D_{r,n,q},U,W)$ is a BEL-configuration in $V(rn,q)$;
\item
$\cB(U) \cap \cB(W) = \emptyset$;
\item
$U \cap \tilde{\cB}(W) = \emptyset$;
\item
$W \cap \tilde{\cB}(U) = \emptyset$;
\item
No element of $U^{\times}$ is an $\Fqn$-multiple of an element of $W^{\times}$;
\end{enumerate}
\end{lemma}

We can embed $V(rn,q)$ in $V(rn+n,q)$ in a natural way, that is by letting $V(rn+n,q) = \{(v_0,v_1,\ldots,v_r):v_i \in \Fqn\}$, and taking $V(rn,q)$ to be the subspace of vectors with $v_0=0$. The spread $\D_{r,n,q}$ naturally extends to the Desarguesian spread $\D_{r+1,n,q}$ of $V(rn+n,q)$.

In \cite{BEL2007}, the following was shown.
\begin{theorem}[\emph{BEL-construction}]
\label{thm:BEL}
Let $(\D,U,W)$ be a BEL-configuration in $V(rn,q)\leq V(rn+n,q)$. Choose an element $v$ of $V(rn+n,q)\backslash V(rn,q)$.

Then the set
\[
\cS(\D,U,W) := \left\{\frac{\langle \cB(z),W \rangle}{W} : z \in \langle v,U \rangle\right\}
\] 
is a semifield spread in $\frac{V(rn+n,q)}{W}\cong V(2n,q)$.
\end{theorem}

\begin{remark}
\label{rem:belop1}
It is clear from the definition, as was shown in \cite{BEL2007}, that for any $\phi \in \GammaL(rn,q)$, the semifield spread $\cS(\D^{\phi},U^{\phi},W^{\phi})$ is equivalent to $\cS(\D,U,W)$. 

If $\D=\D_{r,n,q}$, then the group $\GammaL(r,q^n)$, embedded naturally in $\GammaL(rn,q)$, stabilizes $\D$, see e.g. \cite{LaVdV20??}. Hence  (\cite{BEL2007}, Theorem 2.4), if $\phi \in \GammaL(r,q^n)$, then $(\D,U,W)$ is a BEL-configuration if and only if $(\D,U^{\phi},W^{\phi})$ is a BEL-configuration, the spreads $\cS(\D,U,W)$ and $\cS(\D,U^{\phi},W^{\phi})$ are equivalent, and the corresponding semifields are isotopic.
So if we fix the Desarguesian spread $\D=\D_{r,n,q}$, we can define equivalence classes 
\begin{eqnarray}\label{eqn:equiv_classes}
[U,W] := \{(U^{\phi},W^{\phi}):\phi \in \GammaL(r,q^n)\}.
\end{eqnarray}
We will return to these classes in the subsequent sections.
\end{remark}

\begin{definition}
\label{def:BEL-equivalence}
Two BEL-configurations $(\D,U,W)$ and $(\D',U',W')$ are called {\it equivalent}, notation
$(\D,U,W)\equiv (\D',U',W')$, if and only if the spreads $\cS(\D,U,W)$ and $\cS(\D',U',W')$ are equivalent.
\end{definition}

We will sometimes denote this spread $\cS(\D,U,W)$ simply by $\cS(U,W)$. Now we will show how to find a semifield multiplication which gives a spread equivalent to $\cS(U,W)$. Such an explicit multiplication was only obtained in the case $r=n$ in \cite[ Theorem 4.1]{BEL2007}, and this was done in order to prove that every semifield can be constructed from a BEL-configuration in $V(n^2,q)$.
First we find a convenient way to express $U$ and $W$.

\begin{lemma}\label{lem:Uf}
Any subspace $U$ of $V(rn,q)$ of dimension $n$ is equal to
\begin{eqnarray}\label{eqn:U_f}
U_f := \{(f_1(x),f_2(x),\ldots,f_r(x)):x \in \Fqn\}
\end{eqnarray}
for some $r$-tuple $f = (f_1,f_2,\ldots,f_r)$ of $\Fq$-endomorphisms of $\Fqn$.

Any subspace $W$ of $V(rn,q)$ of dimension $rn-n$ is equal to
\begin{eqnarray}\label{eqn:W_g}
W_g := \{(x_1,x_2,\ldots,x_r):x_i \in \Fqn \mid \sum_{i=1}^r g_i(x_i)=0\}.
\end{eqnarray}
for some $r$-tuple $g = (g_1,g_2,\ldots,g_r)$ of $\Fq$-endomorphisms of $\Fqn$.
\end{lemma}

\begin{proof}
As $\dim_{\Fq}(U)= \dim_{\Fq}(\Fqn)=n$, $U$ and $\Fqn$ are isomorphic as $\Fq$-vector spaces, and hence $U$ is the image of some injective $\Fq$-linear map, say $F$, from $\Fqn$ into $V(rn,q)$. Let $e_i$ be the element of $(\Fqn)^r$ with $1$ in the $i$-th position and zeroes everywhere else, and let $V_i = \{\alpha e_i:\alpha \in \Fqn\}$. Then if we let $f_i(x)e_i$ denote the projection of $F(x)$ onto $V_i$, i.e. $F(x) = \sum_i f_i(x)e_i$, then $U = U_f$, as claimed.

Now $\dim_{\Fq}(W)= rn-n$, and so $W$ is the kernel of some surjective map $G$ from $V(rn,q)$ onto $\Fqn$. If we define $g_i:\Fqn \rightarrow \Fqn:x \mapsto G(xe_i)$, then for $v = \sum_i x_i e_i$, we obtain 
$G(v) =\sum_i G(x_i e_i)= \sum_i g_i(x_i)$, and $W=W_g$, as claimed.
\end{proof}

Note that if we denote $fA = (f_1A,f_2A,\ldots,f_rA)$ and $Bg = (Bg_1,Bg_2,\ldots,Bg_r)$ for any $A,B \in GL(n,\Fq)$, then clearly $U_{fA} = U_f$ and $W_{Bg} = W_g$. Note also that for any $r$-tuple $f$ (resp. $g$) of endomorphisms,  $\dim(U_f)=n$ (resp. $\dim(W_g)=rn-n$) if and only if the map $F$ (resp. $G$) as defined in the proof of Lemma \ref{lem:Uf}
 is injective (resp. surjective).

One of the steps in the BEL-construction consists of taking the quotient space of $W$ in $V(rn+n,q)$, where we again assume $V(rn,q)$ is embedded in $V(rn+n,q)$ as before.
To that extent we consider the following homomorphism $T_g$ which will have kernel $W_g$.
Given an $r$-tuple of endomorphisms $g = (g_1,g_2,\ldots,g_r)$, define a map $T_g : V(rn+n,q) \rightarrow V(2n,q)$ by
\[
T_g(x_0,x_1,\ldots,x_r) = (x_0,\sum_{i=1}^r g_i(x_i)).
\]
Clearly $W_g = \ker(T_g)$, and $\im(T_g) = V(2n,q)$. By the First Isomorphism Theorem for vector spaces, we have
$\frac{V(rn+n,q)}{W_g} \cong V(2n,q)$, and we denote this isomorphism by $\theta$:
\begin{equation}
\label{eqn:Tgisom}
\theta~:~\frac{V(rn+n,q)}{W_g} \rightarrow V(2n,q)~:~ v + W_g \mapsto T_g(v)
\end{equation}

\begin{theorem}
\label{thm:fgmult}
Suppose $(\D_{r,n,q},U_f,W_g)$ is a BEL-configuration in $V(rn,q)$, for some $r$-tuples of endomorphisms $f = (f_1,\ldots,f_r)$, and $g = (g_1,\ldots,g_r)$.
The semifield spread $\cS(\D_{r,n,q},U_f,W_g)$ is equivalent to $\D(\SSS_{f,g})$, where $\SSS_{f,g}$ is defined as the (pre)semifield with multiplication
\begin{eqnarray}
\label{eqn:sfgmult}
\SSS_{f,g}(x,y) := \sum_{i=1}^r g_i(f_i(x)y).
\end{eqnarray}

\end{theorem}

\begin{proof}

Put $v = (1,0,\ldots,0)\in V(rn+n,q)\setminus V(rn,q)$. Using the isomorphism from (\ref{eqn:Tgisom}), we have that for all $z\in \langle v,U_f\rangle$, 
\[
\left ( \frac{\langle \cB(z),W_g \rangle}{W_g} \right )^\theta =  T_g(\cB(z)),
\]
and
\[
\cS(U_f,W_g) \simeq \left\{T_g(\cB(z)\rangle) : z \in \langle v,U_f \rangle\right\}.
\]

 Suppose first that $z \in \langle v, U_f\rangle \backslash U_f$. Then we may take $z = v+u$ for some unique $u \in U_f$. Hence $z  = (1,f_1(x),\ldots,f_r(x))$ for some $x \in \Fqn$, and so
\[
\cB(v+u) = \{(y,f_1(x)y,\ldots,f_r(x)y):y \in \Fqn\}.
\]
Now if $z =u \in U_f$, then 
\[
\cB(u) = \{(0,f_1(x)y,\ldots,f_r(x)y):y \in \Fqn\}.
\]
Hence, for $u \in U_f$ 
\begin{align*}
T_g(\cB(u+v)) &= \{(y,\sum_{i=1}^r g_i(f_i(x)y):y \in \Fqn\}\\
T_g(\cB(u)) &= \{(0,\sum_{i=1}^r g_i(f_i(x)y):y \in \Fqn\}.
\end{align*}

By Theorem \ref{thm:BEL}, this gives a semifield spread. We can see that this is precisely equal to the spread $\cS(\SSS_{f,g})$, where $\SSS_{f,g}$ is as defined in (\ref{eqn:sfgmult}), and $\cS(\SSS_{f,g})$ is as defined in (\ref{def:semspread}).
\end{proof}

\begin{example}
Let $f = (1,c^{1/\beta}\alpha/\beta)$, $g=(1,-\beta)$. Then
\[
\SSS_{f,g}(x,y) = xy-((c^{1/\beta}x^{\alpha/\beta})y)^\beta = xy-cx^{\alpha}y^{\beta},
\]
so $\SSS_{f,g} = \GTF_{c,\alpha,\beta}$, and hence $\cS(\D_{r,n,q},U_f,W_g)$ is equivalent to  $\cS(\mathrm{GTF}_{c,\alpha,\beta})$.
\end{example}

\begin{remark}
Recall from (\ref{eqn:rightmult2}) that every (pre)semifield multiplication can be written as
\[
\SSS(x,y) = \sum_{i = 0}^{n-1} l_i(x)y^{q^i},
\]
for some $l_i$. Then taking $g_i(z)=z^{q^i}$, and $f_i(x) = l_i(x)^{q^{-i}}$, we get that
\[
\SSS_{f,g}(x,y) = \SSS(x,y),
\]
and so $\cS(\SSS) \simeq \cS(\D_{r,n,q},U_f,W_g)$.

The triple $(\D_{r,n,q},U_f,W_g)$, with these particular $n$-tuples of endomorphisms $f_i$ and $g_i$, is precisely the BEL-configuration which was used in \cite{BEL2007}, Theorem 4.1, to show that every (pre)semifield can be constructed from some BEL-configuration in $V(n^2,q)$. So for $r=n$, the approach using the representation of the subspaces as $U_f$ and $W_g$, is the same as in the proof of \cite[Theorem 4.1]{BEL2007}. But the advantage of the viewpoint taken here is that this explicit multiplication can not only be obtained for $r=n$, but for any $r\geq 2$.
\end{remark}

\begin{remark}
As pointed out in the proof of Theorem \ref{thm:fgmult}, it follows from the BEL-construction (Theorem \ref{thm:BEL}) that $\cS(\D_{r,n,q},U_f,W_g)$ forms a semifield spread, and this was proven geometrically in \cite{BEL2007}.
It then follows that $\SSS_{f,g}$ is a (pre)semifield, with multiplication defined in (\ref{eqn:sfgmult}).
However, with the explicit descriptions of the subspaces $U_f$ and $W_g$, defined by the $r$-tuples of endomorphisms $f_i$ and $g_i$, we are able to prove directly that the multiplication (\ref{eqn:sfgmult}) 
defines a presemifield. 
\begin{proof}
It is clear that this multiplication is bilinear in $x$ and $y$ over $\Fq$, and hence it only remains to show that $\SSS_{f,g}$ has no zero divisors. Suppose that $\sum_{i=1}^r g_i(f_i(x)y)=0$ for some $x,y \ne 0$. Then $(f_1(x)y,\ldots,f_r(x)y)\in W_g$. Hence $(f_1(x),\ldots,f_r(x))\in \tilde{\cB}(W_g) \cap U_f$, a contradiction on $(\D_{r,n,q},U_f,W_g)$ being a BEL-configuration by Lemma \ref{lem:belprop}.
\end{proof}
Consequently, this implies that $\cS(\SSS_{f,g})$ is a semifield spread.
This provides an alternative proof of the fact that the BEL-construction gives a semifield spread. 
\end{remark}

We know from \cite{BEL2007} that every presemifield can be constructed from a BEL-configuration in $V(n^2,q)$, and we have seen above that the generalized twisted fields can be constructed for $r=2$, and there are 
many other examples. For instance, whenever we have a multiplication where some $l_i(x)$ or $r_j(y)$, as defined in (\ref{eqn:rightmult2}), are zero, we can find a BEL-configuration for $r<n$, see also \cite[Remark 4.2]{BEL2007}.

The following lemma gives a geometric sufficient condition to obtain a BEL-configuration for smaller $r$.

\begin{theorem}
\label{lem:Wcontain}
Suppose $(\D,U,W)$ is a BEL-configuration in $V(rn,q)$. If $W$ contains an element of $\D$, then there exists a BEL-configuration $(\D',U',W')$ in $V((r-1)n,q)$ such that $\cS(\D,U,W) = \cS(\D',U',W')$.
\end{theorem}

\begin{proof}
Let $\D=\D_{r,n,q}$, $U=U_f$, $W=W_g$, so $\cS(\D,U,W) \simeq \cS(\SSS_{f,g})$. Suppose $\cB(v) \leq W_g$.
Then $\sum g_i(v_ix)=0$ for all $x\in \Fqn$. Without loss of generality, we may assume
 that $v_k =1$ for some $k\in \{1,\ldots, r\}$. Then $g_k(x) = -\sum_{i \ne k} g_i(v_ix)$ for all $x\in \Fqn$. So
\[
\SSS_{f,g}(x,y) = \sum_{i \ne k} g_i((f_i(x)-v_i f_k(x))y).
\]
Defining $f'_i(x) = f_i(x)-v_i f_k(x)$, and $g' = (g_1,\ldots,g_{k-1},g_{k+1},\ldots,g_r)$, $f' = (f'_1,\ldots,f'_{k-1},f'_{k+1},\ldots,f'_r)$, we get that $\SSS_{f,g}=\SSS_{f',g'}$, and hence $\cS(\SSS_{f,g})=\cS(\SSS_{f',g'})$. It follows that
$(\D_{r-1,n,q},U_{f'},W_{g'})$ is a BEL-configuration in $V((r-1)n,q)$ satisfying $\cS(\D,U,W) = \cS(\D_{r-1,n,q},U_{f'},W_{g'})$.
\end{proof}

%This question will be addressed in \cite{Belrank}.

%\begin{remark}
%In \cite{Lavrauw2011}, it was noted that the BEL-construction also works if we allow $U$ to be a subspace over any subfield of $\Fqn$, possibly smaller than $\Fq$, of the appropriate dimension. Note that in the geometric language of that paper, this is called an \emph{$\FF_{q_0}$-linear set}. If $U$ is an $\FF_{q_0}$-subspace, then each $f_i$ is an $\FF_{q_0}$-endomorphism of $\Fqn$.
%
%The multiplication then is still as in (\ref{eqn:sfgmult}). It is not difficult to see that the right nucleus contains $\Fq$, and the centre contains $\FF_{q_0}$.
%\end{remark}

\section{BEL-configurations and polarities}

In this section we will show how to find a BEL-configuration for the dual spread of $\cS(\D,U,W)$. As a corollary, we will get a nice description of symplectic semifields in the context of BEL-configurations.

Let $\rho$ be any polarity on $V(rn,q)$, and let $b_{\rho}$ denote its associated nondegenerate (symmetric bilinear, symplectic, or hermitian) form. For a spread $\D$, define the set $\D_{\rho}$ as follows:
\begin{eqnarray}\label{eqn:Drho}
\D_{\rho} = \{H^{\rho}: \dim(H) = rn-n \textrm{ and 
$H$ is generated by elements of $\D$}\}.
\end{eqnarray}
\begin{lemma}
\label{lem:Drho}
Let $\D$ be a Desarguesian spread in $V(rn,q)$, and $\rho$ any nondegenerate polarity. Then $\D_{\rho}$ is again a Desarguesian spread.
\end{lemma}

\begin{proof}
We assume without loss of generality that $\D=\D_{r,n,q}$. Then the elements of $\D$ correspond to $1$-dimensional $\Fqn$-subspaces of $(\Fqn)^r$, and so $kn$-dimensional subspaces generated by elements of $\D$ correspond to $k$-dimensional $\Fqn$-subspaces of $(\Fqn)^r$. For a formalized description of this ``field reduction'' see \cite{LaVdV20??}.

Let $H_1^{\rho}$, $H_2^{\rho}$ be distinct elements of $\D_{\rho}$. 
Any intersection of subspaces spanned by elements of a Desarguesian spread is partitioned by elements of that spread. Hence $\dim(\langle H_1,H_2 \rangle) = rn$, and so $\dim(H_1^{\rho}\cap H_2^{\rho}) = 0$, proving that $\D_{\rho}$ is a spread.

Suppose $r=2$. In this case $\D_{\rho} = \D^\rho=\{S^{\rho}:S \in \D\}$, and the proof easily follows from the
fact that the dual of a semifield spread of $V(2n,q)$ corresponds to the transpose of the semifield, and the Desarguesian spread corresponds to the field.

Suppose $r>2$. Then $\D_{\rho}$ is Desarguesian if and only if every subspace spanned by elements of $\D_{\rho}$ is partitioned by elements of $\D_{\rho}$. Let $S = \langle H_1^{\rho},\ldots,H_k^{\rho}\rangle$, with $\dim(S) = nk$. Then $S^{\rho} = \bigcap_{i=1..k}H_i$, and $\dim(S^{\rho}) = n(r-k)$. Now there are precisely $\frac{q^{nk}-1}{q^n-1}$ subspaces $H_i$ such that $S^{\rho} \leq H_i$ and $H_i$ is an $(rn-n)$-dimensional subspace spanned by elements of $\D$. But then $H_i^{\rho} \leq S$ for all $i$, and so the set $\{H_i^{\rho}:i \in {1, \ldots, \frac{q^{nk}-1}{q^n-1}}\}\subset \D_{\rho}$ partitions $S$, proving that $\D_{\rho}$ is Desarguesian.
\end{proof}

\begin{theorem}
\label{thm:BELpolarity}
Suppose $\rho$ is a nondegenerate polarity in $V(rn,q)$, and  $\D=\D_{r,n,q}$. Then
$(\D,U,W)$ is a BEL-configuration if and only if $(\D_{\rho},W^{\rho},U^{\rho})$ is a BEL-configuration.
\end{theorem}

\begin{proof}
Suppose $(\D_{\rho},W^{\rho},U^{\rho})$ is not a BEL-configuration. Then there exists some $H^{\rho} \in \D_{\rho}$ such that $H^{\rho}\cap U^{\rho}$ and $H^{\rho} \cap W^{\rho}$ are both non-trivial. Now
\[
H \cap U = H^{\rho \rho }\cap U^{\rho \rho} = \langle H^{\rho},U^{\rho} \rangle^{\rho},
\]
and so $\dim(H\cap U) = \dim(\langle H^{\rho},U^{\rho} \rangle^{\rho}) > 0$. By the definition of $\D_{\rho}$, $H$ is spanned by elements of $\D$, and hence there exists an element $T\in \D$ such that $T\subset H$ and $T\in \cB_\D(U)$. On the other hand, since $H^\rho \in \cB_{\D^\rho}(W^\rho)$, we have
$\dim(H\cap W) >rn-2n$. 
This implies that each element of $\D$ which is contained in $H$, must intersect $W$ non-trivially, i.e. must belong
to $\cB_\D(W)$. This implies that also $T\in \cB_\D(W)$, a contradiction, as $(\D,U,W)$ is a BEL-configuration. If follows that $(\D_{\rho},W^{\rho},U^{\rho})$ is a BEL-configuration. 
For the converse it suffices to note that $(\D_\rho)_\rho=\D$.
\end{proof}

Hence, by Theorem \ref{thm:BELpolarity}, from a semifield spread $\cS(\D,U,W)$ and a polarity $\rho$ we obtain a semifield spread $\cS(\D_{\rho},W^{\rho},U^{\rho})$. However, as we will show, this operation does not induce an extension of the Knuth orbit. First we will show that different choices of the polarity $\rho$ give equivalent semifield spreads. Subsequently, we will show that this operation corresponds to the dual operation on spreads, or in other words the transpose operation on semifields.

Define a nondegenerate symmetric bilinear form $b$ on $V(rn,q)$ by
\begin{eqnarray}
\label{eqn:bilform}
b(u,v) = b((u_1,\ldots,u_r),(v_1,\ldots,v_r)) := \tr\left(\sum_{i=1}^r u_i v_i\right),
\end{eqnarray}
where $\tr$ denotes the field trace from $\Fqn$ to $\Fq$. Denote the corresponding polarity by $\perp$, i.e. if $M$ is any subspace of $V(rn,q)$, then
\begin{eqnarray}
\label{eqn:perp}
M^{\perp} = \{v:v \in V(rn,q)\mid b(u,v)=0 ~\forall u \in M\}.
\end{eqnarray}
%It is well known that $\rank(M^{\perp}) = rn-\rank(M)$.

\begin{lemma}\label{lem:D_perp}
If $\D=\D_{r,n,q}$ and $\perp$ denotes the polarity as in (\ref{eqn:perp}), then $\D_\perp$=$\D$.
\end{lemma}
\begin{proof}
If $H$ is a subspace of dimension $rn-n$ spanned by elements of $\D$, then $H^{\perp} \in \D$, and so $\D_{\perp} = \D$.
\end{proof}

\begin{theorem}\label{thm:rho=perp}
Let $\D=\D_{r,n,q}$. If $\rho$ is a nondegenerate  polarity of $V(rn,q)$, then the BEL-configurations $(\D_{\rho},W^{\rho},U^{\rho})$ and $(\D,W^{\perp},U^{\perp})$ are equivalent.
\end{theorem}

\begin{proof}
First we claim that there exists some $\phi \in \GammaL(rn,q)$ such that
\[
b_{\rho}(u,v) = b(u,v^{\phi})
\]
for all $u,v \in V(rn,q)$, where $b$ is as defined in (\ref{eqn:bilform}) and $b_\rho$ is the form associated to the polarity $\rho$. If we take the elements of $V(rn,q)$ to be column vectors of length $rn$ over $\Fq$, then for every symmetric bilinear, symplectic or hermitian form $b_{\rho}$ on $V(rn,q)$ there exists an invertible $rn \times rn$ matrix $B_{\rho}$ with entries in $\Fq$ and an $\Fq$-automorphism $\sigma$ such that $b_{\rho}(u,v) = u^T B_{\rho} v^{\sigma}$. Hence defining $\phi \in \GammaL(rn,q)$ by $v^{\phi} = B_{\perp}^{-1}B_{\rho}v^{\sigma}$, where $B_\perp$ is the matrix corresponding to $b$ as defined in (\ref{eqn:bilform}), the claim is proven.

Hence for any subspace $M$, we have that $v \in M^{\rho}$ if and only if $b(u,v^{\phi})=0$ for all $u \in M$, if and only if $v^{\phi} \in M^{\perp}$, implying $M^{\rho} = M^{\perp \phi^{-1}}$. 
It also follows that $\D_\rho=(\D_\perp)^{\phi^{-1}}$. Moreover, by Lemma \ref{lem:D_perp}, $\D_{\perp} = \D$. Then $(\D_{\rho},W^{\rho},U^{\rho}) = ((\D_{\perp})^{\phi^{-1}},W^{\perp\phi^{-1}},U^{\perp\phi^{-1}}) = (\D^{\phi^{-1}},W^{\perp\phi^{-1}},U^{\perp\phi^{-1}})$. Hence by Remark \ref{rem:belop1}, we can apply $\phi$ to get that $\cS(\D_{\rho},W^{\rho},U^{\rho})$ is equivalent to $\cS(\D,W^{\perp},U^{\perp})$, as claimed.
\end{proof}

\begin{lemma}
\label{lem:Ufperp}
Consider the subspace $U_f$ of $V(rn,q)$ as in Lemma \ref{lem:Uf}
for some $r$-tuple $f = (f_1,f_2,\ldots,f_r)$ of $\Fq$-endomorphisms of $\Fqn$.
Then $U_f^{\perp} = W_{\hat{f}}$, where $\hat{f} := (\hat{f}_1, \ldots,\hat{f}_{r})$, and 
\[
W_{\hat{f}}:=\{(x_1,x_2,\ldots,x_r): x_i\in \Fqn ~|~ \sum_i \hat{f_i}(x_i) = 0\}.
\]
\end{lemma}

\begin{proof}
Now $v \in U_f^{\perp}$ if and only if
\begin{align*}
b((f_1(u),\ldots,f_r(u)),v) &= \tr(\sum_i f_i(u)v_i)\\
		&= \tr(u \sum_i \hat{f_i}(v_i)) = 0
\end{align*}
for all $u\in \Fqn$. But this occurs if and only if $\sum_i \hat{f_i}(v_i) = 0$, i.e. if and only if $v \in W_{\hat{f}}$, proving the claim. 
\end{proof}

\begin{lemma}
\label{lem:Sfgtranspose}
Suppose $f$ and $g$ are $r$-tuples of endomorphisms such that $\SSS_{f,g}$ is a semifield, with multiplication as defined in (\ref{eqn:sfgmult}). Then $\SSS_{f,g}^t = \SSS_{\hat{g},\hat{f}}$.
\end{lemma}

\begin{proof}
If $R_y$ denotes the endomorphism of right multiplication by $y$ in $\SSS_{f,g}$, then 
\[
R_y = \sum_{i} g_i y f_i.
\]

Hence $\hat{R}_y = \sum_i \hat{f}_i \hat{y} \hat{g}_i$. But $\hat{y}=y$, and so
\begin{align*}
\SSS_{f,g}^t(x,y) &= \sum_i (\hat{f_i}y \hat{g_i})(x)\\
					&= \sum_i \hat{f}_i(\hat{g}_i(x)y)\\
					&= \SSS_{\hat{g},\hat{f}}(x,y),
\end{align*}
as claimed.
\end{proof} 

\begin{theorem}
\label{thm:BELtranspose}
Let $\D=\D_{r,n,q}$, $\rho$ a nondegenerate polarity and $(\D,U,W)$ a BEL-configuration in $V(rn,q)$.
Then the semifield spreads $\cS(\D_\rho,W^{\rho},U^{\rho})$ and $\cS(\D,U,W)$ are dual to each other.
\end{theorem}

\begin{proof}
%Suppose $\dim(U)=n$, $\dim(W) =rn$, and $(\D,U,W)$ is \emph{not} a BEL-configuration. Then there exists some $S \in D$ such that $\dim(S\cap U)\geq 1$ and $\dim(S\cap W) \geq 1$. 
%
%Now $\dim(S^{\perp} \cap U^{\perp}) \geq \dim(\langle S,U \rangle^{\perp}) \geq (r-2)n+1$, since $\dim(\langle S,U \rangle) \leq \dim(S)+\dim(U)-1=2n-1$. Similarly, $\dim(S^{\perp} \cap W^{\perp}) \geq \dim(\langle S,U \rangle^{\perp}) \geq 1$. 
%
%But $S^{\perp}$ is partitioned by elements of $\D$ (proof?), and hence there exists some $T\in \D$ intersecting $W^{\perp}$ and contained in $S^{\perp}$. But then $T$ must intersect $S^{\perp} \cap U^{\perp}$ nontrivially, and hence intersect $U^{\perp}$ nontrivially, and so $(\D,W^{\perp},U^{\perp})$ is not a BEL-configuration. 
%
%Repeating the argument for $(\D,W^{\perp},U^{\perp})$, we get that $(\D,U,W)$ is a BEL-configuration if and only if $(\D,W^{\perp},U^{\perp})$ is a BEL-configuration.

Let $(U,W) = (U_f,W_g)$ for some $f,g$. By Theorem \ref{thm:fgmult}, $\cS(\D,U,W)$ is equivalent to $\cS(\SSS_{f,g})$. By Lemma \ref{lem:Ufperp}, $\cS(\D,W^{\perp},U^{\perp}) = \cS(\D,U_{\hat{g}},W_{\hat{f}})$, which is equivalent to $\cS(\SSS_{{\hat{g},\hat{f}}})$. By Lemma \ref{lem:Sfgtranspose}, $\SSS_{f,g}^t = \SSS_{\hat{g},\hat{f}}$, and hence $\cS(\SSS_{f,g}^t)=\cS(\SSS_{\hat{g},\hat{f}})$, proving that $\cS(\D,W^{\perp},U^{\perp})$ is equivalent to $\cS(\SSS_{f,g}^t)$. It follows that
$\cS(\D,W^{\perp},U^{\perp})$ and 
$\cS(\D,U,W)$ are dual to each other. 
Theorem \ref{thm:rho=perp} completes the proof.
\end{proof}

\begin{remark}
We may define now an operation on the equivalence classes $[U,W]$ defined in Remark \ref{rem:belop1}, where the Desarguesian spread is taken to be $\D_{r,n,q}$. Define 
\begin{eqnarray}\label{eqn:[U,W]^t}
[U,W]^{t} = [W^{\perp},U^{\perp}].
\end{eqnarray}
Then this operation is well defined, as $U^{\perp\phi}=U^{\phi\perp}$ for all $\phi \in \GammaL(r,q^n)$.
\end{remark}

We are now ready to prove a new characterization of symplectic semifield spreads, using BEL-configurations.

\begin{theorem}
\label{thm:sympl}
Suppose $(\D_{r,n,q},U,U^{\perp})$ is a BEL-configuration in $V(rn,q)$, where $\perp$ is as defined in (\ref{eqn:perp}). Then $\cS(\D_{r,n,q},U,U^{\perp})$ is a symplectic semifield spread. Conversely, for every symplectic semifield spread $\cS'$, there exists some $r$ and some $U$ a subspace of dimension $n$ in $V(rn,q)$ such that $\cS'$ is equivalent to $\cS(\D_{r,n,q},U,U^{\perp})$.
\end{theorem}

\begin{proof}
The first statement follows immediately from Theorem \ref{thm:BELtranspose}. For the converse, suppose $\cS'$ is a symplectic semifield spread. Then $\cS'$ is equivalent to $\cS(\SSS)$, where $\SSS^{dtd}$ is some commutative presemifield. Suppose $\SSS^{dtd}(x,y) = \sum_{i,j} c_{ij} x^{q^i}y^{q^j}$, and let $C$ be the $n \times n$ matrix over $\Fqn$ with $(i,j)$-th entry $c_{ij}$. Since $\SSS^{dtd}$ is commutative, $C$ is a symmetric matrix. Hence there exist some column vectors $v_k \in (\Fqn)^n$, $k=1,\ldots,r$,  such that
\[
C = \sum_{k=1}^r v_k v_k^T.
\]
In other words, we write the symmetric matrix $C$ as a sum of $r$ symmetric matrices of rank one. If we denote the $i$-th coordinate of $v_k$ by $f_{ki}$, then we have that $c_{ij} = \sum_k f_{ki}f_{kj}$, and so
\[
\SSS^{dtd}(x,y) = \sum_{k=1}^r\sum_{i,j = 0}^{n-1} f_{ki}f_{kj}x^{q^i}y^{q^j} = \sum_{k=1}^r\left(\sum_{i = 0}^{n-1} f_{ki}x^{q^i}\right)\left(\sum_{i = 0}^{n-1}f_{kj}y^{q^j}\right).
\]
Now if we define endomorphisms $f_k$ by $f_k(x) = \sum_{i=0}^{n-1} f_{ki}x^{q^i}$, it is clear that
\[
\SSS^{dtd}(x,y) = \sum_{k=1}^r f_k(x)f_k(y).
\]
Hence
\begin{align*}
\SSS(x,y) &= \sum_{k=1}^r \hat{f}_k(f_k(x)y)\\
			&= \SSS_{f,\hat{f}}(x,y),
\end{align*}
and so, by Theorem \ref{thm:fgmult} and Lemma \ref{lem:Ufperp}, $\cS(\SSS)=\cS(\SSS_{f,\hat{f}})$ is equivalent to $\cS(\D,U_f,U_f^{\perp})$, completing the proof.
\end{proof}

\begin{remark}
Note that we do not assume in Theorem \ref{thm:sympl} that $r\leq n$, in contrast to Theorem 4.1 of \cite{BEL2007}. This is because the proof relies on writing a symmetric $n\times n$ matrix as the sum of $r$ symmetric $n \times n$ matrices of rank one of a specific type, that is of the form $vv^T$ for some column vector $v$. This decomposition is not necessarily possible if we impose the condition $r\leq n$.
\end{remark}

\section{Isotopy and BEL-configurations}
%
%It is well known that $\GammaL(r,q^n)$, embedded naturally in $\GammaL(rn,q)$, is precisely the subgroup which permutes the elements of $\D$. Hence clearly if $\phi \in \GammaL(r,q^n)$, then $(U,W)$ is a Bel configuration if and only if $(U^{\phi},W^{\phi})$ is a Bel configuration. In \cite{BEL2007}, the following was shown.
%\begin{theorem}
%For any $\phi \in \GammaL(r,q^n)$ and any Bel configuration $(U,W)$, we have that $S(U,W)$ is isotopic to $S(U^{\phi},W^{\phi})$.
%\end{theorem}
%Hence we can define equivalence classes $[U,W] = \{(U^{\phi},W^{\phi}):\phi \in \GammaL(r,q^n)\}$. We will return to these classes in Section ??.

Now we will consider further operations on BEL-configurations which preserve the BEL property. This section will follow \cite{Lavrauw2008}, where the case $r=n$ was considered. Afterwards, we will show how, when $r<n$, we can use this operation to produce non-isotopic semifields from a single BEL-configuration.

It is clear from the BEL property (Lemma \ref{lem:belprop}), that if $\phi,\phi'$ are elements of $\GammaL(rn,q)$ which fix the set $\tilde{\cB}_{\D}(W)$, then $(\D,U,W)$ is a BEL-configuration if and only if $(\D,U^{\phi},W^{\phi'})$ is a BEL-configuration. We now consider the question of when the two semifields arising from these configurations are isotopic, or in other words when the two semifield spreads are equivalent.

For the remainder of this section, we will set $\D=\D_{r,n,q}$.

Suppose $W_g$ is a subspace of $V(rn,q)$ of dimension $rn-n$ as defined in (\ref{eqn:W_g}). Define a map $\psi_g$ from $V(rn,q)$ into $\End_{\Fq}(\Fqn)$ by
\[
\psi_g:(v_1,\ldots,v_r) \mapsto \sum g_i v_i,
\]
where $\sum g_i v_i$ means the endomorphism $z \mapsto \sum g_i (v_i z)$.

\begin{lemma}
\label{lem:psiinj}
The map $\psi_g$ is injective if and only if $W_g$ does not contain an element of $\D$.
\end{lemma}
\begin{proof}
The map $\psi_g$ is not injective if and only if there exists some $v \ne 0$ such that $\psi_g(v)$ is the zero map. This occurs if and only if $\sum_i g_i(v_i x) =0$ for all $x \in \Fqn$, if and only if $\cB(v) = \{(v_1 x,\ldots,v_r x) :x \in \Fqn\}\leq W_g$, proving the claim.
\end{proof}

Note that this condition is the same condition from Lemma \ref{lem:Wcontain}.

\begin{lemma}
Let $v \in V(rn,q)^{\times}$. Then $\psi_g(v)$ is a singular endomorphism if and only if $v \in \tilde{\cB}(W_g)$.
\end{lemma} 

\begin{proof}
We have that $\psi_g(v)$ is singular, if and only if there exists some $\alpha \in \Fqn^{\times}$ such that $\psi_g(v)(\alpha) = 0$, if and only if  there exists some $\alpha \in \Fqn^{\times}$ such that $(\alpha v_1,\ldots,\alpha v_r) \in W_g$, if and only if  $v \in \tilde{\cB}(W_g)$, proving the claim.

%$v\in \tilde{\cB}(W)$ if and only if there exist $w = (w_1,\ldots,w_r)$, and $\alpha \in \Fqn^{\times}$, such that $\sum g_i(x_i)=0$ and $v=\alpha w$.
%
%Now $\phi_g(v)(\alpha^{-1}) = \sum g_i(\alpha \alpha ^{-1}x_i) = \sum g_i(x_i) = 0$, and so $\phi_g(v)$ is singular.
%
%Conversely, suppose $\phi_g(v)$ is singular. Then there exists some $\alpha^{-1} \in \Fqn^{\times}$ such that $\phi_g(v)(\alpha)=0$. Then $\sum g_i(v_i \alpha^-1)=0$. Define $w = \alpha^{-1}v$. Then clearly $w \in W_g$, and so $v \in \tilde{\cB}(W_g)$, completing the proof.
\end{proof}

\begin{lemma}
\label{lem:psig}
Suppose $(\D,U_f,W_g)$ is a BEL-configuration. Then $\psi_g(U_f)$ is a subspace of $\End_{\Fq}(\Fqn)$ of dimension $n$, in which all of its nonzero elements are nonsingular. Furthermore, $\psi_g(U_f)$ is precisely the semifield spread set $R(\SSS_{f,g}^d)$.
\end{lemma}

\begin{proof}
This first part follows immediately from the previous lemma. Now
\begin{align*}
\psi_g(U_f) &= \{y \mapsto \sum g_i(f_i(x)y):x \in \Fqn\}\\
			&= \{y \mapsto \SSS_{f,g}(x,y):x \in \Fqn\}\\
			&= \{L_x : x \in \Fqn\}\\
			&= R(\SSS_{f,g}^d),
\end{align*}
proving the claim.
\end{proof}

\begin{theorem}
Suppose $W$ is a subspace of $V(n^2,q)$ of dimension $n^2-n$ such that $W$ does not contain an element of $\D$. Suppose $(\D,U,W)$ and $(\D,U',W)$ are two BEL-configurations. Then $\cS(\D,U',W)$ is equivalent to $\cS(\D,U,W)$ or the dual of $\cS(\D,U,W)$ if and only if there exists some $\phi \in \mathrm{Stab}(\tilde{\cB}(W)) \leq \GammaL(n^2,q)$ such that $U' = U^{\phi}$.
\end{theorem}

\begin{proof}
By Lemma \ref{lem:Uf}, we may write $U=U_f$, $U' = U_h$, and $W=W_g$. By Lemma \ref{lem:psig}, we can identify $V(n^2,q)$ with $\End_{\Fq}(\Fqn)$ via the map $\psi_g$, which is injective by Lemma \ref{lem:psiinj}. Furthermore, the set of singular endomorphisms is precisely the set $\psi_g(\tilde{\cB}(W_g))$. Then $\psi_g(U_f)$ and $\psi_g(U_h)$ are subspaces corresponding to the semifield spread sets defined by $\SSS_{f,g}^d$ and $\SSS_{h,g}^d$. 

A map $\phi$ is in $\mathrm{Stab}(\tilde{\cB}(W_g))$ if and only if $\psi_g \phi \psi_g^{-1}$ fixes the set of singular endomorphism in $\End_{\Fq}(\Fqn)$. 

Now the invertible semilinear maps preserving the set of singular endomorphisms are precisely those of the form
\[
A \mapsto XA^{\sigma}Y
\] 
or 
\[
A \mapsto X\hat{A}^{\sigma} Y
\]
for some nonsingular endomorphisms $X,Y$, and some $\sigma \in \Aut(\Fq)$. 

Then
\begin{center}
$\cS(\D,U_h,W_g) \simeq \cS(\D,U_f,W_g)$ or the dual of $\cS(\D,U_f,W_g)$ \\
$\Leftrightarrow$\\
$\cS(\SSS_{h,g})\simeq \cS(\SSS_{f,g})$ or $\cS(\SSS_{f,g}^t)$\\
$\Leftrightarrow$\\
$\SSS_{h,g}$ is isotopic to $\SSS_{f,g}$ or $\SSS_{f,g}^t$ \\
$\Leftrightarrow$\\
$\SSS_{h,g}^d$ is isotopic to $\SSS_{f,g}^d$ or $\SSS_{f,g}^{td}$ \\
$\Leftrightarrow$\\
there exist invertible endomorphisms $X,Y$ and some $\sigma \in \Aut(\Fq)$ such that $\psi_g(U_f) = X\psi_g(U_h)^{\sigma}Y$ or $\psi_g(U_f) = X\widehat{\psi_g(U_h)}^{\sigma}Y$\\
$\Leftrightarrow$\\
 there exists some $\rho \in \GammaL(\End_{\Fq}(\Fqn),q)$ fixing the set of singular endomorphisms such that $\psi_g(U_f) = \rho \psi_g(U_h)$\\
$\Leftrightarrow$\\
there is a map $\phi = \psi_g^{-1} \rho^{-1} \psi_g \in \GammaL(n^2,q)$ fixing $\tilde{\cB}(W)$ such that $U_h=U_f^{\phi}$.
\end{center}

\end{proof}

\begin{remark}
This theorem and its proof closely follow \cite{Lavrauw2011}, Theorem 17. That theorem is stated for a specific $W = W_g$, where $g = (1,\sigma,\ldots,\sigma^{n-1})$ and $\sigma(x) = x^q$ for all $x \in \Fqn$. The extension to arbitrary $W$ is trivial. However there is a small error in that theorem, which we rectify here. 

The inaccuracy occurs on page 909: in the line ``In particular, $\varphi$ fixes both families of maximal subspaces of the Segre variety, and hence, by Corollary 12, $\phi$ fixes $\cB(W)$.'', the $\cB(W)$ should instead read $\tilde{\cB}(W)$. 

For the converse then, again $\cB(W)$ should be replaced by $\tilde{\cB}(W)$, and hence rather than implying an isotopy between the two semifields under consideration, it implies that the first is either isotopic to the second, or isotopic to the transpose of the second.

The above proof is self contained, and does not rely on Theorem 17 of \cite{Lavrauw2011}, though the method is the same.
\end{remark}

\section{BEL-configurations in $V(2n,q)$}

\subsection{Action of the stabilizer of $\tilde{B}(W)$}

We showed in the previous section that if $(\D,U,W)$ is a BEL-configuration in $V(n^2,q)$, and $W$ does not contain an element of $\D$, then the elements of the set $\{\cS(\D,U^{\phi},W):\phi \in \mathrm{Stab}\tilde{B}(W)\}$ are all equivalent to $\cS(\D,U,W)$ or $\cS(\D,U,W)^{\epsilon}$, and so we do not get an extension of the Knuth orbit in this case.

However, when $r<n$, the situation is different. Firstly, it is not clear whether if $\cS(\D,U,W)$ and $\cS(\D,U',W)$ are equivalent then there must exist a $\phi$ fixing $\tilde{\cB}(W)$ such that $U'=U^{\phi}$. Conversely, we will now present an example in $V(2n,q)$ where $\cS(\D,U^{\phi},W)$ is not equivalent to any Knuth derivative of $\cS(\D,U,W)$. The example will concern generalized twisted fields.

%\begin{remark}
%Note that we are essentially looking at a subspace of the space of endomorphisms, and asking what linear maps can preserve the set of singular elements contained therein. Then with two maximal subspaces of nonsingular endomorphisms contained, the questions then are: (i) if there is a linear singular preserver of the full endomorphism space mapping one to the other, is there necessarily a linear singular preserver of the subspace mapping one to the other?; and (ii) if there is a linear singular preserver of subspace mapping one to the other, is there necessarily a linear singular preserver of the the full endomorphism space mapping one to the other?
%
%The above implies that for the case $r=n$, the answer to both questions is ``yes''. In the case $r<n$, we present an example where the answer to (ii) is ``no''. Question (i) for $r<n$ is as yet unresolved.
%\end{remark}

%...
%
%
%$\{g_1,g_2,\ldots,g_r\}$ are right-independent over $\Fqn$. 

Let $r=2$, and let $g = (1,-\beta)$, where $\beta \in \Aut(\Fqn)$, and $\beta$ has fixed field $\FF_{q^{n/t}}$. Put
\begin{eqnarray}
W := W_g = \{(x^{\beta},x):x \in \Fqn\}.
\end{eqnarray}
In \cite{LaShZa2013}, this subspace and the set $\tilde{\cB}(W)$ was investigated, and the following properties proved. Recall that we are assuming the Desarguesian spread $\D=\D_{2,n,q}$.

In what follows $N$ denotes the norm function from $\Fqn$ to $\FF_{q^{n/t}}$, i.e. $N(a) = aa^{\beta}\ldots a^{\beta^{t-1}}$.

\begin{theorem}
(i) $\tilde{\cB}(W)$ is a nonsingular hypersurface of degree $n$, defined by
\[
\tilde{\cB}(W) = \{(a,b):a,b \in \Fqn \mid N(a)=N(b) \ne 0\}.
\]
(ii) The subspaces of dimension $n$ contained in $\tilde{\cB}(W)$ are precisely those of the form $S_{\gamma,k} := \{(x,kx^{\gamma}):x \in \Fqn\}$, where $k\in \Fqn$, $N(k)=1$ and $\gamma \in \Aut(\Fqn:\FF_{q^{n/t}})$. These can be divided into $t$ systems which each partition $\tilde{\cB}(W)$, by defining $\mathcal{S}_{\gamma} := \{S_{\gamma,k}: k\in \Fqn | N(k)=1\}$ for each $\gamma \in \Aut(\Fqn:\FF_{q^{n/t}})$.

(iii) Let $G$ denote the setwise stabilizer of $\tilde{\cB}(W)$ in $\GammaL(2n,q)$. Then $G$ contains the elements
\begin{align*}
\phi_{k,m,\gamma,\delta}:(a,b) &\mapsto (ka^{\gamma},mb^{\delta});\\
\phi'_{k,m,\gamma,\delta}:(a,b) &\mapsto (kb^{\gamma},ma^{\delta})
\end{align*}
for $\gamma,\delta \in \Aut(\Fqn:\FF_{q^{n/t}})$, $k,m \in \Fqn$ such that $N(k)=N(m) \ne 0$. These form a subgroup of $G$. If $n$ is prime, then $G$ contains no other elements.
\end{theorem}

So we can consider BEL-configurations of the form $(\D,U,W)$ in $V(2n,q)$, and investigate the action of $G$. If $U=U_f$, then $\cS(\D,U_f,W)$ is equivalent to $\cS(\SSS_{f,g})$, and
\[
\SSS_{f,g}(x,y) = f_1(x)y -(f_2(x)y)^{\beta}.
\]

Let $\phi= \phi_{1,1,1,\gamma}$, that is $\phi((a,b)) = (a,b^{\gamma})$, where $\gamma \in \Aut(\Fqn:\Fq)$. By the above, $\phi \in G$.  Then $U_f^{\phi} = U_h$, with $h = (f_1,\gamma \circ f_2)$. Hence $\cS(\D,U_f^{\phi},W) = \cS(\D,U_h,W)$, which is equivalent to $\cS(\SSS_{h,g})$, and
\[
\SSS_{h,g}(x,y) = f_1(x)y -(f_2(x)^{\gamma}y)^{\beta}.
\]
Now let $f= (1,c^{1/\beta} \alpha/\beta)$, so $U_f= \{(x,c^{1/\beta} x^{\alpha/\beta}):x \in \Fqn\}$. Then
\[
\SSS_{f,g}(x,y) = xy-cx^{\alpha}y^{\beta},
\]
and so $\SSS_{f,g}(x,y) = \mathrm{GTF}_{c,\alpha,\beta}$. Then 
\[
\SSS_{h,g} = xy-c^{\gamma}x^{\gamma \alpha}y^{\beta},
\]
and so $\SSS_{h,g} = \GTF_{c^{\gamma},\alpha\gamma,\beta}$.

It is clear from Theorem \ref{thm:isotopism_GTF} that in general these are nonisotopic (pre)semifields. In fact, if $\gamma \alpha = \beta$, then we get a presemifield isotopic to a finite field.

So this action can indeed produce nonisotopic semifields, or equivalently nonequivalent semifield spreads.

%For all other elements of $G$, it is not difficult to show that no other nonisotopic semifields can be obtained in this way.

Note that we may also allow $G$ to act on $W_g$, and still retain the BEL property. If we take $\phi = \phi_{k,m,\gamma,\delta}$, $\phi' = \phi_{k',m',\gamma',\delta'}$, then after some calculations we see that 

\begin{align}
\cS(\D,U_f,W_g) &\simeq \cS(\mathrm{GTF}(c,\alpha,\beta))\\
\cS(\D,U_f^{\phi},W_g^{\phi'}) &\simeq \cS(\mathrm{GTF}(c',\alpha',\beta'))
\end{align}

where $\alpha' = \frac{\alpha\gamma'\delta}{\delta'\gamma}$, $\beta' = \frac{\beta\gamma'}{\delta'}$, and
\[
c' = \frac{k'}{m'}\frac{m^{\beta'}}{k^{\alpha'}}c^{\beta'\delta/\beta}.
\]

Hence from the single BEL-configuration $(\D,U_f,W_g)$ corresponding to the generalized twisted field $\GTF_{c,\alpha,\beta}$, we can produce all generalized twisted fields $\GTF_{c',\alpha',\beta'}$ such that $N(c)=N(c')$, as well as the finite field.

\subsection{A group of order $8$ acting on equivalence classes $[U,W]$}
\label{sec:belrank2}
In \cite{BEL2007}, it was noted that if $(\D,U,W)$ is a BEL-configuration in $V(2n,q)$, then $\dim(U)=\dim(W) = n$, and condition 2 of Lemma \ref{lem:belprop} becomes symmetric in $U$ and $W$. Hence we have the following.
\begin{lemma}
In $V(2n,q)$, $(\D,U,W)$ is a BEL-configuration if and only if $(\D,W,U)$ is a BEL-configuration.
\end{lemma}
This operation is known as \emph{switching}. We denote $\cS(\D,W,U) = \cS(\D,U,W)^s$. Taking $\D=\D_{r,n,q}$, we similarly define an action on the equivalence classes defined in Remark \ref{rem:belop1} by $[U,W]^s := [W,U]$. Note that it remains to be shown that this is well defined on equivalence classes of semifield spreads.

Together with the transpose operation defined in (\ref{eqn:[U,W]^t}), this gives a group of order $4$ acting on equivalence classes $[U,W]$. In this section, we will define a new operation $e$, which will extend this to a group of order $8$.

First, we show how we can compute a multiplication corresponding to $\cS(\D,W,U)$. This was calculated in \cite{BEL2007} only for the class of semifields which are two dimensional over a nucleus. We will again assume $\D=\D_{r,n,q}$ for the remainder of this section.

\begin{theorem}
Let $\D=\D_{r,n,q}$. Suppose $(\D,U,W)$ is a BEL-configuration in $V(2n,q)$. Then there exist $a,b \in \End_{\Fq}(\Fqn)$ such that $\cS(\D,U,W)$ is equivalent to $\cS(\SSS)$, where
\[
\SSS(x,y) = xy+b(a(x)y),
\]
and $\cS(\D,U,W)^s$ is equivalent to $\cS(\SSS')$, where
\[
\SSS'(x,y) = xy+a(b(x)y).
\]
\end{theorem}

\begin{proof}
Recall that $\cS(\D,U,W) \simeq \cS(\D,U^{\phi},W^{\phi})$ for any $\phi \in \GammaL(2,q^n)$. Hence without loss of generality, we may assume that $U$ intersects $S_{\infty} = \{(0,x):x \in \Fqn\}$ trivially, and $W$ intersects $S_0 = \{(x,0):x \in \Fqn\}$ trivially. Thus we may assume $U= U_{(1,a)}$, and $W = W_{(1,b)}$ for some endomorphisms $a,b$. Hence by Theorem \ref{thm:fgmult}, $\cS(\D,U,W)$ is equivalent to $\cS(\SSS_{(1,a),(1,b)}) = \cS(\SSS)$.

Now it is straightforward to check that $U=W_{(a,-1)}$ and $W=U_{(-b,1)}$. Hence by Theorem \ref{thm:fgmult}, $\cS(\D,U,W)^s = \cS(\D,W,U)$ is equivalent to $\cS(\SSS_{(-b,1),(a,-1)}) = \cS(\SSS')$, completing the proof.
\end{proof}

Hence in order to classify semifields which can be constructed from a BEL-configuration with $r=2$, it suffices to consider pairs of endomorphisms $(a,b)$ such that
\[
\frac{b(a(x)y)}{xy} \ne -1 ~ \forall x,y \in \Fqn^{\times}.
\]

Note that if $q>2$, we may also assume that $U$ intersects $S_0$ trivially and $W$ intersects $S_{\infty}$ trivially, and hence that $a$ and $b$ are invertible. A simple counting argument shows that if $q>2$, there exist two elements of $\D$ not contained in $B(U)\cup B(W)$, and as $\GammaL(2,q^n)$ acts $2$-transitively on $\D$, the assertion holds. 

Now we will define a new operation on pairs of subspaces of dimension $n$ in $V(2n,q)$ which preserves the BEL property. Consider condition 5 of Lemma \ref{lem:belprop}, which says that no nonzero vectors $u \in U$ and $w \in W$ are $\Fqn$-multiples of each other. This implies that
\[
\det \npmatrix{u_1&u_2\\w_1&w_2} \ne 0
\]
for all $0 \ne u=(u_1,u_2) \in U$, $0 \ne w=(w_1,w_2) \in W$.  

As before, we may take $U=U_{(1,a)}$, $W=U_{(b,-1)}$. Then
\[
\det \npmatrix{u_1&u_2\\w_1&w_2} = \det \npmatrix{x&a(x)\\b(y)&-y} = -(xy+a(x)b(y)) \ne 0
\]
for all $x,y \in \Fqn^{\times}$. Hence this above formula defines a (pre)semifield multiplication.

Removing the minus sign, and taking the transpose-dual of this, we get that
\[
\SSS'(x,y) = xy + \hat{a}(b(x)y)
\]
defines a (pre)semifield multiplication.

Recall the symplectic polarity $\epsilon$ defined on $V(2n,q)$ in (\ref{def:epsilon}). Then it can be seen that $U^{\epsilon} = U_{(1,\hat{a})} = W_{(\hat{a},-1)}$, and so $\cS(\SSS') \simeq \cS(\D,W,U^{\epsilon})$, and $\cS(\D,W,U^{\epsilon})$ is a BEL-configuration. Hence, switching $W$ and $U^{\epsilon}$, we get that $(\D,U^{\epsilon},W)$ is a BEL-configuration, and we have proved the following.

\begin{theorem}
Suppose $(\D,U,W)$ is a BEL-configuration in $V(2n,q)$. Then $(\D,U^{\epsilon},W)$ is also a BEL-configuration.
\end{theorem}

As $M^{\phi\epsilon}=M^{\epsilon\phi}$ for all $\phi \in \GammaL(2,q^n)$, we can define an operation $e$ on the equivalence classes defined in Remark \ref{rem:belop1} by $[U,W] \mapsto [U^{\epsilon},W]$.

Hence we have shown that the following three actions preserve the BEL property.
\begin{align*}
t &:[U,W] \mapsto [W^{\perp},U^{\perp}]\\
s &:[U,W] \mapsto [W,U]\\
e &:[U,W] \mapsto [U^{\epsilon},W].
\end{align*}
Each of these operations are involutions, and together they define a group of order $8$, as they satisfy $ese=t$, $st=ts$. Hence from a single BEL-configuration in $V(2n,q)$, we can produce up to $4$ Knuth orbits, and hence up to $24$ isotopy classes.

We now present a table containing a multiplication representing each of these isotopy classes. We will let $U=U_{(1,a)}=W_{(a,-1)}$, and $W = W_{(1,b)} = U_{(b,-1)}$. 

\begin{center}

\begin{tabular}{|c|c|c|c|c|}
\hline
$\SSS$ & $\id$ &$s$& $e$ & $es$\\
\hline
$id$& $xy+b(a(x)y)$ & $xy+a(b(x)y)$ & $xy+b(\hat{a}(x)y)$ & $xy+\hat{b}(\hat{a}(x)y)$\\
\hline
$t$& $xy+\hat{a}(\hat{b}(x)y)$ & $xy+\hat{a}(b(x)y)$ & $xy+a(\hat{b}(x)y)$ & $xy+\hat{b}(a(x)y)$\\
\hline
$d$& $xy+b(xa(y))$ & $xy+a(xb(y))$ & $xy+b(x\hat{a}(y))$ & $xy+\hat{a}(xb(y))$\\
\hline
$td$& $xy+\hat{a}(x\hat{b}(y))$ & $xy+\hat{b}(x\hat{a}(y))$ & $xy+a(x\hat{b}(y))$ & $xy+\hat{b}(xa(y))$\\
\hline
$dt$& $xy+\hat{b}(x)a(y)$ & $xy+\hat{a}(x)b(y)$ & $xy+\hat{b}(x)\hat{a}(y)$ & $xy+a(x)b(y)$\\
\hline
$dtd$& $xy+a(x)\hat{b}(y)$ & $xy+b(x)\hat{a}(y)$ & $xy+\hat{a}(x)\hat{b}(y)$ & $xy+b(x)a(y)$\\
\hline
\end{tabular}
\end{center}

Here the element in the top row is applied first, e.g. the multiplication in the third row, second column corresponds to $(S(U,W)^s)^d$. Note that $(S(U,W)^d)^s$ is undefined, as $s$ (and $e$) are only defined on BEL-configurations. To incorporate $d$ to form a single group action, we would require a way to find $(U',W')$ such that $S(U',W') \simeq S(U,W)^d$. As yet we have no such method.

\begin{example}
Let $b(x) = -x^{\beta}$, $a(x)= c^{1/\beta}x^{\alpha/\beta}$. 

\begin{center}
\begin{tabular}{|c|c|c|c|c|}
\hline
$\SSS$ & $\id$ &$s$& $e$ & $es$\\
\hline
$id$& $(c,\alpha,\beta)$ & $(c^{1/\beta},\alpha,\alpha/\beta)$ & $(c^{\beta/\alpha},\beta^2/\alpha,\beta)$ & $(c^{1/\alpha},\beta^2/\alpha,\beta/\alpha)$\\
\hline
$t$& $(c^{1/\alpha},1/\alpha,\beta/\alpha)$ & $(c^{1/\alpha\beta},1/\alpha,1/\beta)$ & $(c^{1/\beta},\alpha/\beta^2,\alpha/\beta)$ & $(c^{1/\beta^2},\alpha/\beta^2,1/\beta)$\\

\hline
\end{tabular}

\end{center}

Comparing this table with Theorem \ref{thm:isotopism_GTF}, it is not difficult to see that $(c,\alpha,\beta)$ can be chosen such that these 8 presemifields are pairwise non-isotopic. Furthermore, calculating the full Knuth orbits as in Example 1, we see that these operations can indeed produce 24 different isotopy classes from a single BEL-configuration.

\end{example}

\section{Rank two semifields}

Rank two semifields, that is, semifields which are two dimensional over a nucleus have received particular focus in recent years, due to their close connections to various objects in finite geometry. See for example \cite{LaPo2011}, Section 3 for an overview.

Suppose $n=2m$ is even, and suppose $\SSS$ is two dimensional over its right nucleus. Then we can assume
\[
\SSS(x,y) = f_1(x)y - (f_2(x)y)^{q^m}
\]
for some endomorphisms $f_1,f_2$. Then $\SSS=\SSS_{f,g}$, where $g = (1,x \mapsto -x^{q^m})$. Hence every rank two semifield is of the form $\SSS_{f,g}$ for some $f$, and so can be constructed from a BEL-configuration $(\D,U_f,W_g)$ in $V(2n,q)$. In this case, $\im(\psi_g) = \End_{\Fqm}(\Fqn)$, and  $\psi_g(\tilde{\cB}(W_g))$ is precisely the set of singular $\Fqm$-endomorphisms of $\Fqn$, and corresponds to a hyperbolic quadric in a projective $3$-dimensional space over $\Fqm$. 

Hence classifying BEL-configurations $(\D,U,W_g)$ in this case is precisely the same as classifying linear sets of rank $n$ skew from a hyperbolic quadric in $\PG(3,q^m)$. This is the approach used in for example \cite{CAPOTR2004}, \cite{JOMAPOTR2008}, \cite{MaPoTr2007}, \cite{EbMaPoTr2009} to classify rank two semifields of dimension $4$ and $6$ over their centre. 

So, as was shown in \cite{BEL2007}, the switching operation can be applied to rank two semifields. In \cite{Lavrauw2005}, it was shown that in this special case, switching has a geometric interpretation (``dualising an ovoid''). In \cite{Kantor2009} it was shown that in this case, switching is well defined up to isotopism. In \cite{BaBr2004} it was shown that this operation is, in general, non-trivial. In \cite{LUMAPOTR2008} this operation was christened the \emph{translation dual}.

Now we show that the operation $e$ defined in Section \ref{sec:belrank2} turns out to be trivial in this case. For the case $W = W_{(1,x \mapsto x^{q^m})}$, we have that $W^{\epsilon} = W$. For let $u,v \in W$. Then there exist $x,y \in \Fqn$ such that $u=(x,x^{q^m})$, $v = (y,y^{q^m})$. Then
\begin{align*}
b_{\epsilon}(u,v) &= \tr(xy^{q^m} - x^{q^m}y)\\
					&= \tr(xy^{q^m}) - \tr(xy^{q^{n-m}})\\
					&= 0,
\end{align*}
since $n=2m$. Hence when $W=W_{(1,x \mapsto -x^{q^m})}$,
\[
[U,W]^e = [U,W^{\epsilon}] = [U,W],
\]
implying that the operation $e$ is trivial in the case of rank two semifields. However, as seen in Section \ref{sec:belrank2}, this operation is not always trivial.

\bibliographystyle{plain}

\end{document}